\documentclass{amsart}
\usepackage{amssymb}
\usepackage{enumerate}
\usepackage{hyperref}

\newcommand{\ds}{\displaystyle} 
\renewcommand{\phi}{\ensuremath{\varphi}}
\newcommand{\Q}{\ensuremath{\mathbb{Q}}}
\newcommand{\Z}{\ensuremath{\mathbb{Z}}}
\newcommand{\C}{\ensuremath{\mathbb{C}}}
\newcommand{\cA}{\ensuremath{\mathcal{A}}}
\newcommand{\into}{\ensuremath{\hookrightarrow}} 
				
\DeclareMathOperator{\Hilb}{Hilb}
 		
\DeclareMathOperator{\rk}{rk}				
\DeclareMathOperator{\gr}{gr}
\DeclareMathOperator{\nbc}{nbc}

\newtheorem{lem}{Lemma}[section]
\newtheorem{thm}[lem]{Theorem}

\theoremstyle{definition}
\newtheorem{ex}[lem]{Example}
\theoremstyle{remark}
\newtheorem{rmk}[lem]{Remark}

\begin{document}

\title{Cohomology of abelian arrangements}
\author{Christin Bibby}
\address{Department of Mathematics, University of Oregon, Eugene, Oregon 97403}
\curraddr{Department of Mathematics, Western University, London, Ontario, Canada N6A 5B7}
\email{cbibby2@uwo.ca}
\thanks{Partially supported by NSF grant DMS-0950383.}
\date{}

\begin{abstract}
An abelian arrangement is a finite set of codimension one abelian subvarieties (possibly translated) in a complex abelian variety. In this paper, we study the cohomology of the complement of an abelian arrangement. For unimodular abelian arrangements, we provide a combinatorial presentation for a differential graded algebra whose cohomology is isomorphic to the rational cohomology of the complement. Moreover, this DGA has a bi-grading that allows us to compute the mixed Hodge numbers.
\end{abstract}

\maketitle

\section{Introduction}\label{introduction}

The goal of this paper is to study the cohomology of the complement of an abelian arrangement. Here, an abelian arrangement is a finite set of codimension one abelian subvarieties (possibly translated) in a complex abelian variety $X$. A special case of this is the configuration space of $n$ points on an elliptic curve $E$. This is the complement of an elliptic version of the braid arrangement in $E^n$, where for $1\leq i<j\leq n$, we have an abelian subvariety $Y_{ij}$ given by the equation $e_i=e_j$. A more general special case is given by any $\ell\times n$ integer matrix, defining an arrangement in $E^n$. Here, each row of the matrix represents an algebraic map $\alpha:E^n\to E$, and we consider the connected components of $\ker\alpha$ for each such $\alpha$.

Totaro \cite{totaro} and Kriz \cite{kriz} each independently studied the cohomology of configuration spaces of smooth complex projective varieties. In particular, their work gives a presentation of a model for the cohomology in our special case of a configuration space on an elliptic curve. 
In this paper, we generalize Totaro's method to compute the rational cohomology of the complement of any abelian arrangement $\cA$ in a complex abelian variety $X$. We denote this complement by $M(\cA)$, and arrive at our results by studying the Leray spectral sequence of the inclusion $f:M(\cA)\into X$. Specifically, we use Hodge theory to show degeneration of this spectral sequence at the $E_3$ term.
Most results stated are valid when considering a complex torus rather than an abelian variety; we need the algebraic structure when discussing the Hodge theory. 

Our results are particularly nice in the case where $\cA$ is unimodular, which means that all multiple intersections of subvarieties in $\cA$ are connected. In this case, we give a presentation of a differential graded algebra $A(\cA)$ in terms of the combinatorics of $\cA$ (the partially ordered set consisting of all intersections of subvarieties in $\cA$). The cohomology of $A(\cA)$  is isomorphic as a graded algebra to the cohomology of $M(\cA)$, by Theorem \ref{DGA}. Moreover, $A(\cA)$ admits a second grading, and it is canonically isomorphic as a bi-graded algebra to $\gr H^*(M(\cA);\Q)$, the associated graded with respect to Deligne's weight filtration. Thus it allows us to compute the mixed Hodge numbers of $M(\cA)$.

\begin{rmk}
While the weight filtration on the cohomology of the complement of a hyperplane or toric arrangement is trivial (by \cite{looijenga}), for an abelian arrangement it is always interesting. For example, consider a punctured elliptic curve $M(\cA)=E\smallsetminus\{p_1,\dots,p_\ell\}$. Here, the first filtered piece of $H^1(M(\cA);\Q)$ consists of the image of the restriction map from $H^1(E;\Q)$, which is neither trivial nor surjective.
This can be seen in the short exact sequence
$$0 \to H^1(E;\Q) \to H^1(E\smallsetminus\{p_1,\dots,p_\ell\};\Q) \to \Q(-1)^{\oplus(\ell-1)}\to0.$$
\end{rmk}

Levin and Varchenko \cite{levinvarchenko} computed cohomology of elliptic arrangements with coefficients in a  nontrivial rank one local system.
Dupont \cite{dupont} also studied the more general case of the complement to a union of smooth hypersurfaces which intersect like hyperplanes in a smooth projective variety. Dupont used a similar but alternative method to that presented in this paper to find the same model for cohomology as described in Section \ref{cohomology}, but he does not give the combinatorial presentation in Section \ref{unimodular}.

In \cite{dupont2}, Dupont uses our decomposition of the Leray spectral sequence in Lemma \ref{decomposition} to show that all toric arrangements are formal. In \cite{suciu}, Suciu uses the model given in Theorem \ref{DGA} to study resonance varieties and formality of elliptic arrangements. 

\bigskip
\textbf{Acknowledgements}. The author would like to thank her advisor Nick Proudfoot for his many helpful comments, suggestions, and guidance. The author would also like to thank Dev Sinha for his advice. 

\section{Preliminaries}\label{preliminaries}

We  consider an arrangement $\cA=\{Y_1,\dots,Y_\ell\}$ of smooth connected divisors in a smooth complex variety $X$, which intersect like hyperplanes. When we say that they \textbf{intersect like hyperplanes}, we mean that for every $p\in X$, there is a neighborhood $U\subseteq X$ of $p$, a neighborhood $V\subseteq T_pX$ of $0$, and a homeomorphism $\phi:U\to V$ that induces $Y_i\cap U\cong T_pY_i\cap V$ for all $Y_i\in \cA$. 

An \textbf{abelian arrangement} is an arrangement $\cA=\{Y_1,\dots,Y_\ell\}$ in an abelian variety $X$ where each $Y_i$ is, up to translation, a codimension-one abelian subvariety. Note that these subvarieties intersect like hyperplanes. 

A \textbf{component} of an arrangement $\cA$ is a connected component of an intersection $Y_S:=\ds\bigcap_{Y\in S}Y$ for some subset $S\subseteq \cA$. Note that the intersections themselves need not be connected. 
We say that the arrangement is \textbf{unimodular} if the intersection $Y_S$ is connected for all subsets $S\subseteq\cA$.
The \textbf{rank} of a component is defined as its complex codimension in $X$. 
Note that for a subset $S\subseteq\cA$ with nonempty intersection $Y_S$, the rank of its components is constant. Hence we define the rank of a subset $S\subseteq\cA$ as the rank of a connected component of $Y_S$, which does not depend on the choice of component. 
If $\rk(S)=|S|$, we say that $S$ is \textbf{independent}. Otherwise, $\rk(S)<|S|$ and we say that $S$ is \textbf{dependent}.
Also let $M(\cA)=X\setminus\cup_{Y\in\cA} Y$ be the complement of the union of the divisors in $\cA$.

\begin{ex}\label{ellipticex}
We describe the motivation behind the terminology using our special case of an elliptic arrangement, where we have $X=E^n$ and an $\ell\times n$ integer matrix. As described in the introduction, each row corresponds to a map $\alpha_i:E^n\to E$. Assume that each row is primitive, so that each $Y_i=\ker\alpha_i$ is a connected abelian subvariety of $X$.  

In this case, an intersection $Y_S$ is the kernel of an $|S|\times n$ submatrix, taking the corresponding rows $\alpha_i$ for $Y_i\in S$. The codimension of $Y_S$ is the rank of the corresponding matrix. In this way, the dependencies of the hyperplanes in $\cA$ correspond to the dependencies of the corresponding $\alpha_i$'s in $\Z^n$.

Further suppose that $\cA$ is a unimodular arrangement and that the rank of the $\ell\times n$ matrix is equal to $n$. Then all $n\times n$ submatrices will have determinant $\pm1$ or $0$. Otherwise, an intersection of subvarieties (that is, the kernel of the corresponding submatrix) would be disconnected. This agrees with the usual notion of a unimodular matrix.
\end{ex}

Let $F$ be a component of the arrangement $\cA$. For any point $p\in F$, define an arrangement $\cA_F^{(p)}$ in the tangent space $T_pX$ consisting of hyperplanes $Y_F^{(p)}:=T_pY$ for all $Y\supseteq F$. If $X$ has complex dimension $n$, then $\cA_F^{(p)}$ is a central hyperplane arrangement in $T_pX\cong \C^n$, and we denote its complement by $M(\cA_F^{(p)})=T_pX\setminus\cup_{Y\supseteq F} Y^{(p)}_F$. This arrangement may be referred to as the localization of $\cA$ at $F$, with respect to the point $p\in F$.

\begin{rmk}\label{localizationrmk} 
We say that a point $p\in F$ is a generic point of $F$ if $p$ is not contained in any smaller component of $\cA$. By our assumption that the divisors intersect like hyperplanes, for a generic point $p\in F$, there is a neighborhood $U\subseteq X$ of $p$ such that $U\cap M(\cA)\cong M(\cA_F^{(p)})$.
\end{rmk}

\begin{rmk}\label{localizationrmk2}
Also by our assumption that the divisors intersect like hyperplanes, the intersection lattice of the arrangement $\cA_F^{(p)}$ does not depend on the choice of $p\in F$. Since the cohomology of $M(\cA_F^{(p)})$ only depends on the combinatorics of $\cA_F^{(p)}$ (by \cite[Theorem~5.90]{orlikterao}), we may write $H^*(M(\cA_F);\Q)$ to mean the cohomology of $M(\cA_F^{(p)})$ for some (any) $p\in F$.

If $\cA$ is an abelian arrangement, then even more can be said. Not only does the cohomology not depend on the choice of $p\in F$, but for any two points $p$ and $q$ of $F$, we have a canonical homeomorphism (via translation) $M(\cA_F^{(p)})\cong M(\cA_F^{(q)})$.
\end{rmk}

\section{Rational Cohomology}\label{cohomology}

Let $\cA=\{Y_1,\dots,Y_\ell\}$ be a set of smooth connected divisors that intersect like hyperplanes in a smooth complex variety $X$, and denote the complement of their union in $X$ by $M(\cA)$.
The inclusion $f:M(\cA)\into X$ gives a Leray spectral sequence of the form
$$E_2^{p,q}=H^p(X;R^qf_*\Q)\implies H^{p+q}(M(\cA);\Q).$$
Recall that $R^qf_*\Q$ is the sheafification of the presheaf on $X$ taking an open set $U$ to $H^q(U\cap M(\cA);\Q)$.

To make use of this spectral sequence, we need to examine the sheaves $R^qf_*\Q$. 
The key is to show that the sheaf $R^qf_*\Q$ is isomorphic to the sheaf $$\displaystyle\bigoplus_{\rk(F)=q} H^q(M(\cA_F);\Q)\otimes (i_F)_*\Q_F,$$ where $i_F:F\into X$ is the closed immersion of a rank-$q$ component $F$. In the proof of the following lemma, we will show this isomorphism by constructing a map on the presheaves which is an isomorphism on stalks. 

\begin{lem}\label{decomposition}
Let $\cA=\{Y_1,\dots,Y_\ell\}$ be a set of smooth connected divisors that intersect like hyperplanes in a smooth complex variety $X$. Then
$$H^p(X;R^qf_*\Q)\cong \bigoplus_{\rk(F)=q}H^p(F;\Q)\otimes H^q(M(\cA_F);\Q).$$
\end{lem}
\begin{proof}
First, we examine the stalks of $R^qf_*\Q$. 
Let $x\in X$. In the following discussion, the localization of $\cA$ at a component will always be with respect to the point $x$, and we will drop the superscript from our notation for localizations. 
Take the unique smallest  component $F_x$ of $\cA$ containing $x$.
Note that $x$ is a generic point of $F_x$, and so for every small enough neighborhood $U$ around $x$, we have $U\cap M(\cA)\cong M(\cA_{F_x})$.
This means that the stalk of our sheaf $R^qf_*\Q$ at $x$ is given by $$H^q(U\cap M(\cA);\Q)\cong H^q(M(\cA_{F_x});\Q).$$

Note that the rank-$q$ components of $\cA_{F_x}$ correspond exactly to the rank-$q$ components of $\cA$ that contain $F_x$.
For such an $F$, we can consider the usual localization of $\cA_{F_x}$ (a central hyperplane arrangement) at the component corresponding to $F$ in $\cA_{F_x}$, denoted by $(\cA_{F_x})_F$. This is the same arrangement as $\cA_F$. Then
Brieskorn's Lemma \cite[p.~27]{brieskorn} implies that $$H^q(M(\cA_{F_x});\Q)\cong\bigoplus_{\substack{F\supseteq F_x\\ \rk(F)=q}} H^q(M( (\cA_{F_x})_F);\Q)\cong \bigoplus_{\substack{F\supseteq F_x\\ \rk(F)=q}} H^q(M(\cA_F);\Q).$$
Since $x$ was a generic point of $F_x$, the rank-$q$ components containing $F_x$ are exactly the rank-$q$ components containing $x$. 
Thus, the stalk at $x\in X$ can be decomposed as 
$$(R^qf_*\Q)_x\cong \bigoplus_{\substack{F\ni x \\ \rk(F)=q}}H^q(M(\cA_F);\Q).$$

Now denote $\epsilon:=\displaystyle\bigoplus_{\rk(F)=q} H^q(M(\cA_F);\Q)\otimes (i_F)_*\Q_F$.
Let $F$ be a component of rank $q$, and consider the summand of $\epsilon$ corresponding to $F$, which we denote by $\epsilon_F$. Note that for an open $U\subseteq X$, we have that $\epsilon_F(U)$ is the direct sum of $H^q(M(\cA_F);\Q)$ over the connected components of $U\cap F$. Moreover, the stalk of $\epsilon$ at $x\in X$ is 
$$\epsilon_x= \bigoplus_{\rk(F)=q}(\epsilon_F)_x\cong \bigoplus_{\substack{\rk(F)=q,\\ x\in F}}H^q(M(\cA_F);\Q).$$

For a rank-$q$ component $F$, consider the arrangement in $X$ defined by $\cA|_F=\{H\in\cA\ |\ H\supseteq F\}$. Note that $F$ is also a component of $\cA|_F$, and every element of $F$ is generic here. So for $x\in F$, take a neighborhood basis $\mathcal{U}_x$ of increasingly small open sets $U$ so that $U\cap M(\cA|_F)$ is homeomorphic to $M((\cA|_F)_F^{(x)})$. 
Together with open sets that don't intersect $F$, these form a basis for $X$. We define a map $\epsilon_F\to R^qf_*\Q$ on this basis, and taking the sum, we will get a morphism $\epsilon\to R^qf_*\Q$. First, if $U\cap F=\emptyset$, then $\epsilon_F(U)=0$ and hence the map on sections is zero. Otherwise, if $U\in\mathcal{U}_x$ for some $x\in F$, then the composition of the inclusion $U\cap M(\cA)\into U\cap M(\cA|_F)$ and the homeomorphism $U\cap M(\cA|_F)\to M((\cA|_F)_F^{(x)})$ induces a map on cohomology $H^q(M(\cA_F);\Q)\to H^q(U\cap M(\cA);\Q)$, giving us a map on sections. 
Now the induced map $\epsilon\to R^qf_*\Q$ is an isomorphism on stalks, hence a sheaf isomorphism. 

Returning to the $E_2$ term of our Leray spectral sequence for the inclusion $f:M(\cA)\into X$, we now have that 
\begin{align*}
H^p(X;R^qf_*\Q)
&\cong H^p(X;\epsilon)\\
&\cong\bigoplus_{\rk(F)=q}H^p(X;H^q(M(\cA_F);\Q)\otimes(i_F)_*\Q_F)\\
&\cong\bigoplus_{\rk(F)=q}H^p(F;\Q)\otimes H^q(M(\cA_F);\Q).
\end{align*}
\end{proof}

If we further take $X$ to be a projective variety, then the $E_2$ term of the spectral sequence is all that is needed to calculate the cohomology of $M(\cA)$. 
\begin{lem}\label{degeneration}
Let $\cA=\{Y_1,\dots,Y_n\}$ be a set of smooth connected divisors that intersect like hyperplanes in a smooth complex projective variety $X$, and denote its complement by $M(\cA)$. Then all differentials $d_j$ in the Leray spectral sequence for the inclusion $f:M(\cA)\into X$ are trivial for $j>2$.
\end{lem}
\begin{proof}
To show that higher differentials are trivial, we consider the weight filtration on $$H^p(X;R^qf_*\Q)\cong \bigoplus_{\rk(F)=q}H^p(F;\Q)\otimes H^q(M(\cA_F);\Q).$$ Note that since $F$ is a smooth complex projective variety, $H^p(F;\Q)$ is pure of weight $p$. Since $M(\cA_F)$ is the complement of a rational hyperplane arrangement, $H^q(M(\cA_F);\Q)$ is pure of weight $2q$ (by \cite{shapiro}).
This implies that $H^p(X;R^qf_*\Q)$ is pure of weight $p+2q$.

Now, the differentials must be strictly compatible with the weight filtration, as explained in the proof of Theorem 3 of \cite{totaro}. Since the $(p,q)$ position on the $E_j$ term will also have weight $p+2q$, the differential $d_j$ will map something of weight $p+2q$ to something of weight $(p+j)+2(q-j+1)=p+2q-j+2$. Being strictly compatible with weights implies that the only nontrivial differential must be when $j=2$.
\end{proof}

Moreover, if we consider only the cohomological grading (by $p+q$) on the $E_2$ term, we have the following theorem.

\begin{thm}\label{cohomologythm}
Let $\cA=\{Y_1,\dots,Y_n\}$ be a set of smooth connected divisors that intersect like hyperplanes in a smooth complex projective variety $X$.
The rational cohomology of $M(\cA)$ is isomorphic as a graded algebra to the cohomology of $E_2(\cA)$ with respect to its differential. 
\end{thm}
\begin{proof}
By Lemma \ref{degeneration}, the Leray spectral sequence degenerates at the $E_3$ term. This implies that the associated graded of $H^*(M(\cA);\Q)$ with respect to the Leray filtration is isomorphic to the cohomology of $E_2(\cA)$.

The groups $E_3^{p,q}=E_\infty^{p,q}$ that contribute to the $k$-th rational cohomology (when $p+q=k$) each have distinct weight (as described in the proof of Lemma \ref{degeneration}), and so the Leray filtration is exactly the weight filtration.
By the work of Deligne \cite[p.~81]{deligne}, the associated graded of $H^*(M(\cA);\Q)$ with respect to its weight filtration is isomorphic to $H^*(M(\cA);\Q)$ as an algebra.
\end{proof}

\begin{rmk}\label{gr}
The $E_2$ term of the spectral sequence forms a differential bi-graded algebra, denoted by $E_2(\cA)$.
The main result of this section was that  we have an isomorphism of algebras 
$$H^*(E_2(\cA))\cong \gr H^*(M(\cA);\Q),$$
where the right hand side is the associated graded with respect to the weight filtration. In particular, if we consider the bi-grading (and not just the cohomological grading) of $E_2(\cA)$, we have
$$H^{p,q}(E_2(\cA))\cong\gr_{p+2q} H^{p+q}(M(\cA);\Q),$$
and we can compute the mixed Hodge numbers of $M(\cA)$. 
\end{rmk}

\begin{rmk}
The same method could be used to study the cohomology of an affine hyperplane arrangement in $\C^n$ or of a toric arrangement in $(\C^\times)^n$.
In fact, this is originally due to Looijenga \cite{looijenga}. 
In these cases, Lemma \ref{decomposition} applies, but Lemma \ref{degeneration} and Theorem \ref{cohomologythm} do not.
\begin{enumerate}
\item Let $\cA=\{H_1,\dots,H_\ell\}$ be an affine arrangement of hyperplanes in a complex affine space $X$ of dimension $n$, and denote the complement $M(\cA)=X\smallsetminus\cup_i H_i$. For the Leray Spectral Sequence of the inclusion $f:M(\cA)\into X$, the $E_2$-term decomposes into $E_2^{0,q}=\oplus_{\rk(F)=q}H^q(M(\cA_F);\Q)$ and $E_2^{p,q}=0$ for $p\neq0$. This forces the differentials to all be trivial, and we see that $E_2(\cA)$ is the Orlik-Solomon algebra $H^*(M(\cA);\Q)$. 
\item Let $\cA=\{T_1,\dots,T_\ell\}$ be an arrangement of codimension-one subtori in a complex torus $X=(\C^\times)^n$, and denote the complement by $M(\cA)=X\smallsetminus\cup_iT_i$. The $E_2$-term of the Leray Spectral Sequence for the inclusion $f:M(\cA)\into T$ decomposes into components, so that $$E_2^{p,q}=\oplus_{\rk(F)=q}H^p(F;\Q)\otimes H^q(M(\cA_F);\Q).$$ Here, $F$ is a complex torus and so $E_2^{p,q}$ is pure of weight $2(p+q)$. Since the differentials $d_j$ respect the weights, $d_j$ must be trivial for all $j$. Thus, $E_2(\cA)\cong\gr_LH^*(M(\cA);\Q)$, the associated graded with respect to the Leray filtration. This decomposition of the cohomology is the decomposition given by De Concini and Procesi in \cite[Remark~4.3]{deconciniprocesi}. In the same paper, De Concini and Procesi also gave a presentation for the cohomology ring in the case of unimodular toric arrangements, and one can see that $H^*(M(\cA);\Q)$ and $E_2(\cA)$ are not isomorphic as algebras. While the Orlik-Solomon relation holds in $E_2(\cA)$, it is a little more complicated in $H^*(M(\cA);\Q)$. 
\end{enumerate}
\end{rmk}

\begin{rmk}
Another interesting result for an abelian arrangement $\cA$ in $X$  comes from considering the deletion and restriction arrangements, with respect to some fixed $Y_0\in\cA$. Here, we mean the analogous notion to the theory of hyperplane arrangements, where the deletion of $Y_0$ is the arrangement $\cA'=\cA\setminus\{Y_0\}$ in $X$ and the restriction to $Y_0$ is the arrangement $\cA''$ in $Y_0$ of all nonempty connected components of $Y\cap Y_0$ where $Y\in\cA'$. 
In the theory of hyperplane arrangements, the long exact sequence of the pair $(M(\cA'),M(\cA))$ relates the cohomologies of $M(\cA)$, $M(\cA')$, and $M(\cA'')$, as follows:
$$\cdots\to H^i(M(\cA'))\to H^i(M(\cA))\to H^{i-1}(M(\cA''))\to H^{i+1}(M(\cA))\to\cdots$$
Moreover, this long exact sequence splits into short exact sequences relating these cohomologies. In the abelian arrangement case, we can get the same kind of long exact sequence. However, it does not split into short exact sequences.
To study the (nontrivial) boundary map, we can consider the long exact sequence induced by a short exact sequence of complexes
$$0\to E_2(\cA')^*\to E_2(\cA)^*\to E_2(\cA'')^{*-1}\to 0.$$
Note that this long exact sequence involves the weight graded quotients, but it is isomorphic to the long exact sequence of the pair.
The boundary map is then seen to be $$\pi_*:H^{i-1}(M(\cA'');\Q)(-1)\to H^{i+1}(M(\cA');\Q)$$ where $\pi:M(\cA'')\into M(\cA')$ is the closed immersion.
\end{rmk}

\begin{rmk}\label{rationalhomotopy}
Dupont \cite{dupont} independently found the same differential graded algebra as described here. He considers the cohomology of the complement of a union $Y=Y_1\cup\cdots\cup Y_\ell$ of smooth hypersurfaces which intersect like hyperplanes in a smooth complex projective variety $X$, and for simplicity he assumes that the arrangement is unimodular. Dupont's method uses the Gysin spectral sequence, which degenerates at the $E_2$ term and has a differential graded algebra $M^*(X,Y)$ as the $E_1$ term. Setting $\cA=\{Y_1,\dots,Y_\ell\}$, the differential graded algebras $E_2(\cA)$ and $M^*(X,Y)$ are isomorphic.
Moreover, Dupont constructs a wonderful compactification of these arrangements, so that the space $X\setminus Y$ can be realized as the complement of a normal crossings divisor $\widetilde{Y}$ in a smooth projective variety $\widetilde{X}$. Dupont also shows functoriality of $M^*$ so that $M^*(X,Y)$ is quasi-isomorphic to $M^*(\widetilde{X},\widetilde{Y})$.

By the work of Morgan \cite{morgan}, the differential graded algebra $M^*(\widetilde{X},\widetilde{Y})$ is a model for the space $X\setminus Y=\widetilde{X}\setminus\widetilde{Y}$, in the sense of rational homotopy theory.
Since our differential graded algebra $E_2(\cA)$ is isomorphic to $M^*(X,Y)$ and hence quasi-isomorphic to $M^*(\widetilde{X},\widetilde{Y})$, $E_2(\cA)$ is a model for the space $M(\cA)=X\setminus Y$.
\end{rmk}

\section{Unimodular Abelian Arrangements}\label{unimodular}

To explicitly describe the $\Q$-algebra structure of the $E_2$ term of the spectral sequence, we assume further that $\cA$ is a unimodular abelian arrangement. Recall that we allow the $Y_i\in\cA$ to be a translation of an abelian subvariety of $X$; denote this subvariety by $\bar{Y}_i$. For each $Y_i\in\cA$, let $E_i=X/\bar{Y}_i$, an elliptic curve, so that $\bar{Y}_i$ is the kernel of the projection $\alpha_i:X\to E_i$.
The $E_2$ term is a bi-graded algebra with a differential, which we denote by $E_2(\cA)$. The $(p,q)$-th graded term is isomorphic to $$\bigoplus_{\rk(F)=q} H^p(F;\Q)\otimes H^q(M(\cA_F);\Q)$$ by Lemma \ref{decomposition}. 

The multiplication of $E_2(\cA)$ can be described as follows:
Let $x_1\otimes y_1$ be in $ H^{p_1}(F_1;\Q)\otimes H^{q_1}(M(\cA_{F_1});\Q)$ and $x_2\otimes y_2$ be in $H^{p_2}(F_2;\Q)\otimes H^{q_2}(M(\cA_{F_2});\Q)$.
If $F_1\cap F_2=\emptyset$ or if $\rk(F_1\cap F_2)\neq q_1+q_2$, then $(x_1\otimes y_1)\cdot(x_2\otimes y_2)=0$.
Otherwise, let $F=F_1\cap F_2$ (which by unimodularity is a component of $\cA$), $p=p_1+p_2$, and $q=q_1+q_2$.
Also let, for $j=1,2$, $\gamma_j:F\into F_j$ and $\eta_j:M(\cA_F)\into M(\cA_{F_j})$ be the natural inclusions.
Then $$(x_1\otimes y_1)\cdot (x_2\otimes y_2)=(-1)^{q_1p_2}(\gamma_1^*(x_1)\cup\gamma_2^*(x_2))\otimes (\eta_1^*(y_1)\cup\eta_2^*(y_2)),$$ an element of $ H^p(F;\Q)\otimes H^q(M(\cA_F);\Q)$.

In particular, consider the case that $F_1=Y_i$ and $F_2=X$. For $1\otimes g$ in $H^0(Y_i;\Q)\otimes H^1(M(\cA_{Y_i});\Q)$ and $x\otimes 1$ in $H^p(X;\Q)\otimes H^0(M(\cA_X);\Q)$, we have $$(1\otimes g)\cdot(x\otimes 1)=(-1)^p\gamma_i^*(x)\otimes g\in H^p(Y_i)\otimes H^q(M(\cA_{Y_i});\Q).$$
Since $Y_i$ is (a possible translation of) the kernel of some map $\alpha_i:X\to E_i$, the kernel of $\gamma_i^*$ contains the image of $\alpha_i^*$ in positive degree. This means that for $p>0$, and any element $x\in H^p(X;\Q)$ that is in the image of $\alpha_i^*$, $(1\otimes g)\cdot(x\otimes 1)=0$. 

We further observe that the row $q=0$ inherits an algebra structure from $H^*(X;\Q)$, and the column $p=0$ inherits an algebra structure from the Orlik-Solomon algebra. In particular, if $\cap_{Y\in \cA}Y\neq\emptyset$, then the column $p=0$ inherits an algebra structure from $H^*(M(\cA_0);\Q)$ where $\cA_0$ is the localization at the intersection of all hyperplanes in $\cA$. 
These algebras are generated in degree one; moreover, they will generate the entire $E_2(\cA)$ algebra. This is because the map $\gamma^*:H^*(X;\Q)\to H^*(F;\Q)$, where $F$ is a component and $\gamma:F\into X$ is the natural inclusion, is surjective.

Since the algebra is generated by $E_2^{1,0}$ and $E_2^{0,1}$, it suffices to describe the differential on $H^0(Y_i;\Q)\otimes H^1(M(\cA_{Y_i});\Q)$ for each $Y_i\in\cA$. This has a canonical generator, since the Orlik-Solomon algebra $H^*(M(\cA_{Y_i});\Q)$ has a canonical generator in degree one. The differential here is determined by the differential of the Leray spectral sequence for the inclusions $X\setminus Y_i\into X$, which takes the generator to $[Y_i]\in H^2(X;\Q)$. 

Now we will describe an algebra $A(\cA)$, determined by the arrangement $\cA$, and prove in Theorem \ref{DGA} that this algebra is isomorphic to $E_2(\cA)$.
Let $B(\cA)=H^*(X;\Q)[g_1,\dots,g_\ell]$, a graded-commutative, bi-graded algebra over $\Q$, where $H^i(X;\Q)$ has degree $(i,0)$ and each $g_j$ has degree $(0,1)$.
Let $I(\cA)$ be the ideal in $B(\cA)$ generated by the following relations:
\begin{enumerate}[(1)]
\item $g_{i_1}\cdots g_{i_k}$ whenever $\cap_{j=1}^k Y_{i_j}=\emptyset$.
\item $\ds\sum_{j=1}^k(-1)^{j-1}g_{i_1}\cdots \hat{g}_{i_j} \cdots g_{i_k}$ whenever $Y_{i_1},\dots,Y_{i_k}$ are dependent.
\item $\alpha_i^*(x)g_i$, where $\alpha_i$ defines $Y_i$ and $x\in H^1(E_i;\Q)$.
\end{enumerate}
For notational purposes, denote $g_C=g_{i_1}\cdots g_{i_k}$ for $C=\{Y_{i_1},\dots,Y_{i_k}\}$ and $\partial g_C=\sum_{j=1}^k(-1)^{j-1}g_{i_1}\cdots\hat{g}_{i_j}\cdots g_{i_k}$.

Let $A(\cA)=B(\cA)/I(\cA)$. Since $I(\cA)$ is homogeneous with respect to the grading on $B(\cA)$, $A(\cA)$ is a bi-graded algebra over $\Q$. Moreover, there is a differential on $A(\cA)$ defined by $dg_i=[Y_i]\in H^2(X;\Q)$ and $dx=0$ for $x\in H^*(X;\Q)$.

\begin{thm}\label{DGA}
Assume that $\cA=\{Y_1,\dots,Y_\ell\}$ is a unimodular abelian arrangement. 
Then there is an isomorphism of bi-graded differential algebras $$\phi:A(\cA)\to E_2(\cA).$$
\end{thm}

Before we prove this theorem, we'll show an example in which this presentation can be used to compute the cohomology of $M(\cA)$.  Moreover, if we consider the bi-grading on $A(\cA)\cong E_2(\cA)$, then we can compute the dimension of $\gr_jH^i(M(\cA);\Q)$, the associated graded with respect to the weight filtration. By Remark \ref{gr}, the $(p,q)$-th graded piece of $H^*(A(\cA))$ will be isomorphic to $\gr_{p+2q}H^{p+q}(M(\cA);\Q)$. We encode the information about dimension in a two-variable polynomial $H(t,u)$, where the coefficient of $t^iu^j$ is the dimension of $\gr_jH^i(M(\cA);\Q)$. 

\begin{ex}
Let $X=E^2$ for an elliptic curve $E$, and let $\alpha_i:E^2\to E$ be projection onto the $i$-th coordinate. Consider the arrangement $\cA=\{Y_1,Y_2,Y_3\}$ in $E^2$ with $Y_1=\ker\alpha_1$, $Y_2=\ker\alpha_2$, and $Y_3=\ker(\alpha_1-\alpha_2)$. Pick generators $x$ and $y$ for $H^*(E;\Q)$, where $xy$ is the class of the identity of $E$. Then $H^*(E^2;\Q)$ is generated by $x_i=\alpha_i^*(x)$ and $y_i=\alpha_i^*(y)$ for $i=1,2$. This then implies that the algebra $B(\cA)$ is the exterior algebra with generators $\{x_1,y_1,x_2,y_2,g_1,g_2,g_3\}$. 

The relations in $I(\cA)$ can be written as 
\begin{enumerate}[(1)]
\item no relations of the type $g_S$ (since all intersections are nonempty)
\item $g_2g_3-g_1g_3+g_1g_2$ (since $\{Y_1,Y_2,Y_3\}$ is minimally dependent)
\item $x_1g_1$, $y_1g_1$, $x_2g_2$, $y_2g_2$, $(x_1-x_2)g_3$, and $(y_1-y_2)g_3$.
\end{enumerate}
The differential of $A(\cA)=B(\cA)/I(\cA)$ is defined by $dx_i=0$, $dy_i=0$, $dg_1=[Y_1]=x_1y_1$, $dg_2=[Y_2]=x_2y_2$, and $dg_3=[Y_3]=x_1y_1-x_1y_2-x_2y_1+x_2y_2$. 

Computing cohomology, the polynomial described above becomes $$H(t,u)=1+4tu+3t^2u^2+2t^2u^3.$$
Setting $u=1$, we obtain the Poincar\'{e} polynomial  $P(t)=1+4t+5t^2$.
\end{ex}

\begin{proof}[Proof of Theorem \ref{DGA}]
First, we show that there is a surjective homomorphism $\phi$, by defining a map $\theta:B(\cA)\to E_2(\cA)$ which induces $\phi$  as follows: 
Let $$\theta(g_i):= 1\otimes e_i\in H^0(Y_i;\Q)\otimes H^1(M(\cA_{Y_i});\Q)$$ 
where $e_i$ is the canonical generator of $H^1(M(\cA_{Y_i});\Q)$,
and for $x\in H^i(X;\Q)$, let $$\theta(x):=x\otimes 1\in H^i(X;\Q)\otimes H^0(M(\cA_X);\Q).$$

We have already observed that $E_2(\cA)$ is generated by $E_2^{1,0}$ and $E_2^{0,1}$. 
Even more explicitly, the elements $1\otimes e_i$ and $x\otimes 1$ as above generate the algebra. Thus, $\theta$ is surjective. 

By our observations above, it is easy to see that $\theta(g_S)=0$ whenever $Y_S=\emptyset$. For relation (2), suppose $S$ is a dependent subset of $\cA$. Then $$\theta(\partial g_S)=1\otimes(\partial e_S)\in H^0(Y_S;\Q)\otimes H^{\rk(S)}(M(\cA_{Y_S});\Q)$$ which is zero since $\partial e_S=0$ in the Orlik-Solomon algebra $H^*(M(\cA_{Y_S});\Q)$. Also, by our observations above, $\theta(\alpha_i^*(x)g_i)$ is equal to zero. Therefore, $\theta(I(\cA))=0$ and hence $\theta$ induces the desired surjection $\phi:A(\cA)\to E_2(\cA)$.

We can decompose $B(\cA)$ with respect to the components of the arrangement, $B(\cA)=\oplus_F B_F$, where $B_F$ is the $\Q$-vector space spanned by $xg_S$ for all standard tuples $S$ of hyperplanes in $\cA$ whose intersection is exactly $F$, and all $x\in H^*(X;\Q)$. The ideal $I(\cA)$ is homogeneous with respect to this grading. Thus, $A(\cA)$ can be decomposed by $A(\cA)=\oplus_F A_F$, where $A_F=B_F/I_F$ with $I_F:=I(\cA)\cap B_F$.

The $E_2$ term of the Leray spectral sequence can also be graded by the components. Here, we have $E_2(\cA)=\oplus_F E_2(F)$, where for each component $F$, $E_2(F)=H^*(F;\Q)\otimes H^{\rk(F)}(M(\cA_F);\Q)$.
Since we will require information about the cohomology of central hyperplane arrangements, let us quickly recall the non-broken circuit basis (see, for example, \cite{orlikterao}). For a given order on the hyperplanes in $\cA$,  a tuple of hyperplanes is called standard if the hyperplanes are written in increasing order. A standard tuple $S$ of hyperplanes is a broken circuit if there is some hyperplane $H$ greater than all those in $S$ such that $S\cup\{H\}$ is a minimally dependent set. 
We say that a standard tuple $S$ is a non-broken circuit if it does not contain any broken circuits. The cohomology of a central hyperplane arrangement has a basis $e_S$ indexed by the non-broken circuits $S$. 

It suffices to show that, as $\Q$-vector spaces, $A_F\cong E_2(F)$. We do this by examining $A_F$. We have $B_F\cong\oplus_S H^*(X;\Q)\cdot g_S$, where the direct sum is taken over all standard tuples $S$ of hyperplanes in $\cA$ with $Y_S=F$. If we consider just the ideal $I_1$ generated by relations (1) and (2), then $$B_F/(I_1\cap B_F)\cong\oplus_S H^*(X;\Q)\cdot g_S,$$ where the sum is taken over all non-broken circuits $S$ with $Y_S=F$. This is because relations (1) and (2) are  just the Orlik-Solomon relations  on the $g_i$'s associated to $F$.

Next, we claim that relation (3) implies that for all $Y_i\supseteq F$, all $S\subseteq\cA$ with $Y_S=F$, and all $x\in H^1(E_i;\Q)$, we have $\alpha_i^*(x)g_S\in I$. This implies that relation (3) depends only on the component $F$, and not on the choice of subset $S$. This claim is clearly true when $Y_i\in S$. If $Y_i\notin S$, then take a maximal independent subset of $S$, denoted by $T$. Then $C:=T\cup\{Y_i\}$ is a dependent set, and $Y_C=Y_T=F$. We may assume, for ease of notation, that our hyperplanes are ordered so that $g_S=g_{(S-T)}g_T$ and $g_C=g_ig_T$. Then we have $g_T-g_i\partial g_T=\partial g_C\in I$, since $C$ is dependent. This implies that 
\begin{align*}
\alpha_i^*(x)g_S
&= \alpha_i^*(x)g_{(S-T)}g_T\\
&= \alpha_i^*(x)g_{(S-T)}(g_T-g_i\partial g_T) + \alpha_i^*(x)g_{(S-T)}g_i\partial g_T\\
&\in I.
\end{align*}

Let $J_F$ be the ideal in $H^*(X;\Q)$ generated by $\alpha_i^*(x)$ for all $Y_i\supseteq F$ and $x\in H^1(E_i;\Q)$. Now, since $H^*(F;\Q)\cong H^*(X;\Q)/J_F$, we must have that $$A_F\cong\oplus H^*(F;\Q)\cdot g_S$$ where the sum is taken over all non-broken circuits $S$ with $Y_S=F$. This is then isomorphic to $H^*(F;\Q)\otimes H^q(M(\cA_F);\Q)\cong E_2(F)$, since the non-broken circuits form a basis for $H^q(M(\cA_F);\Q)$.
\end{proof}

\begin{rmk}
If $\cA$ is not unimodular, we can still define the bi-graded differential algebra $A(\cA)$ and the homomorphism $\phi:A(\cA)\to E_2(\cA)$, but it will no longer be surjective. The problem is that if an intersection $Y_S$ of subvarieties has multiple components, the image of $\phi$ will include the element $1\in H^0(Y_S;\Q)$, but it will not include the corresponding classes for the individual components.
\end{rmk}

\begin{rmk}
While this combinatorial model gives a way of computing cohomology of a unimodular abelian arrangement, it is still unknown if there is a combinatorial formula for the Poincar\'e polynomial (or Betti numbers). However, this model does give us a combinatorial formula for the Euler characteristic (and, more generally, the E-polynomial $H(-1,u)$, which is a specialization of the Hodge polynomial). This works even for non-unimodular arrangements, and it can be realized as a specialization of a Tutte polynomial. 

For a central abelian arrangement, one can associate the Tutte polynomial $$T(x,y) = \sum_{S\subseteq\cA} m(S)(x-1)^{\rk\cA-\rk S}(y-1)^{|S|-\rk S}$$
where $m(S)$ is the number of connected components of the intersection $Y_S$. 
For simplicity, assume that the arrangement is essential (that is, the rank of the arrangement is the dimension of the ambient space, which we will denote by $n$), though the following discussion can be extended to non-essential arrangements. 
Denote the number of non-broken circuits associated to a component $F$ by $\nbc(F)$, so that $\dim H^{\rk(F)}(M(\cA_F);\Q)=\nbc(F)$. We also have that for a component $F$, $\dim H^p(F;\Q) = \binom{2\dim F}{p}$. Thus, using our decomposition in Lemma \ref{decomposition} and an argument similar to that in the toric case \cite[Section~5]{moci}, the Hilbert series of $A(\cA)$ is $$\Hilb_{A(\cA)}(t) = t^n T\left(1+\dfrac{(1+t)^2}{t},0\right).$$
As a corollary (setting $t=-1$), we obtain the explicit formula $$(-1)^n\sum_{F\in\mathcal{P}} \nbc(F)$$ for the Euler characteristic, where $\mathcal{P}$ is the set of dimension-zero components. 

\end{rmk}

\bibliographystyle{amsalpha}
\bibliography{bibby}

\def\cprime{$'$}
\providecommand{\bysame}{\leavevmode\hbox to3em{\hrulefill}\thinspace}
\providecommand{\MR}{\relax\ifhmode\unskip\space\fi MR }
\providecommand{\MRhref}[2]{%
  \href{http://www.ams.org/mathscinet-getitem?mr=#1}{#2}
}
\providecommand{\href}[2]{#2}
\begin{thebibliography}{{Dup}b}

\bibitem[Bri73]{brieskorn}
Egbert Brieskorn, \emph{Sur les groupes de tresses [d'apr\`es {V}. {I}.
  {A}rnol\cprime d]}, S\'eminaire {B}ourbaki, 24\`eme ann\'ee (1971/1972),
  {E}xp. {N}o. 401, Springer, Berlin, 1973, pp.~21--44. Lecture Notes in Math.,
  Vol. 317. \MR{0422674 (54 \#10660)}

\bibitem[DCP05]{deconciniprocesi}
C.~De~Concini and C.~Procesi, \emph{On the geometry of toric arrangements},
  Transform. Groups \textbf{10} (2005), no.~3-4, 387--422. \MR{2183118
  (2006m:32027)}

\bibitem[Del75]{deligne}
Pierre Deligne, \emph{Poids dans la cohomologie des vari\'et\'es
  alg\'ebriques}, Proceedings of the {I}nternational {C}ongress of
  {M}athematicians ({V}ancouver, {B}. {C}., 1974), {V}ol. 1, Canad. Math.
  Congress, Montreal, Que., 1975, pp.~79--85. \MR{0432648 (55 \#5634)}

\bibitem[{Dup}a]{dupont}
C.~{Dupont}, \emph{{Hypersurface arrangements and a global
  Brieskorn-Orlik-Solomon theorem}}, Annales de l'Institut Fourier, to appear.

\bibitem[{Dup}b]{dupont2}
\bysame, \emph{{Purity, formality, and arrangement complements}}, International
  Mathematics Research Notices, to appear.

\bibitem[Kri94]{kriz}
Igor Kriz, \emph{On the rational homotopy type of configuration spaces}, Ann.
  of Math. (2) \textbf{139} (1994), no.~2, 227--237. \MR{1274092 (95c:55012)}

\bibitem[Loo93]{looijenga}
Eduard Looijenga, \emph{Cohomology of {$\mathcal{M}_3$} and
  {$\mathcal{M}^1_3$}}, Mapping class groups and moduli spaces of {R}iemann
  surfaces ({G}\"ottingen, 1991/{S}eattle, {WA}, 1991), Contemp. Math., vol.
  150, Amer. Math. Soc., Providence, RI, 1993, pp.~205--228. \MR{1234266
  (94i:14032)}

\bibitem[LV12]{levinvarchenko}
Andrey Levin and Alexander Varchenko, \emph{Cohomology of the complement to an
  elliptic arrangement}, Configuration Spaces (A.~Bjorner, F.~Cohen,
  C.~Concini, C.~Procesi, and M.~Salvetti, eds.), CRM Series, Scuola Normale
  Superiore, 2012, pp.~373--388 (English).

\bibitem[Moc12]{moci}
Luca Moci, \emph{A {T}utte polynomial for toric arrangements}, Trans. Amer.
  Math. Soc. \textbf{364} (2012), no.~2, 1067--1088. \MR{2846363}

\bibitem[Mor78]{morgan}
John~W. Morgan, \emph{The algebraic topology of smooth algebraic varieties},
  Inst. Hautes \'Etudes Sci. Publ. Math. (1978), no.~48, 137--204. \MR{516917
  (80e:55020)}

\bibitem[OT92]{orlikterao}
Peter Orlik and Hiroaki Terao, \emph{Arrangements of hyperplanes}, Grundlehren
  der Mathematischen Wissenschaften [Fundamental Principles of Mathematical
  Sciences], vol. 300, Springer-Verlag, Berlin, 1992. \MR{1217488 (94e:52014)}

\bibitem[Sha93]{shapiro}
B.~Z. Shapiro, \emph{The mixed {H}odge structure of the complement to an
  arbitrary arrangement of affine complex hyperplanes is pure}, Proc. Amer.
  Math. Soc. \textbf{117} (1993), no.~4, 931--933. \MR{1131042 (93e:32045)}

\bibitem[Suc15]{suciu}
A.~Suciu, \emph{{Around the tangent cone theorem}},
  http://arxiv.org/abs/1502.02279, February 2015.

\bibitem[Tot96]{totaro}
Burt Totaro, \emph{Configuration spaces of algebraic varieties}, Topology
  \textbf{35} (1996), no.~4, 1057--1067. \MR{1404924 (97g:57033)}

\end{thebibliography}

\end{document}